\documentclass[12pt,a4paper,reqno]{amsart}

\usepackage{graphicx}
\usepackage{hyperref}
\usepackage[margin=3cm]{geometry}
\usepackage{mathrsfs}
\usepackage{overpic}
\usepackage[utf8]{inputenc}
\usepackage{amsmath}
\usepackage{amsfonts}
\usepackage{amssymb}
\usepackage{enumerate}
\usepackage{subcaption}
\usepackage{xcolor}

\newcommand{\N}{\ensuremath{\mathbb{N}}}

\newcommand{\R}{\ensuremath{\mathbb{R}}}

\renewcommand{\epsilon}{\varepsilon}
\renewcommand{\geq}{\geqslant}
\renewcommand{\leq}{\leqslant}

\newtheorem{thm}{Theorem}[section]

\newtheorem{lem}[thm]{Lemma}
\newtheorem{prop}[thm]{Proposition}

\newtheorem{rem}[thm]{Remark}

\title[Weak repulsive singularity]
{Periodic bouncing solutions of the Lazer-Solimini equation with weak repulsive singularity}

\author[D. Rojas \and P.J. Torres]{David Rojas$^1$ \and Pedro J. Torres$^2$}

\address{$^1$ Departament d' Informàtica, Matemàtica Aplicada i Estadística, Universitat de Girona, 17003  Girona, Spain}
\email{david.rojas@udg.edu}

\address{$^2$ Departamento de Matemática Aplicada, Universidad de Granada, 18071 Granada, Spain} \email{ptorres@ugr.es}

\subjclass[2010]{Primary: 34C25.}

\keywords{Periodic solution, bouncing, singularity, Poincar\'e-Birkhoff Theorem}

\begin{document}

\begin{abstract}
We prove the existence and multiplicity of periodic solutions of boun\-cing type for a second-order differential equation with a weak repulsive singularity. Such solutions can be catalogued according to the minimal period and the number of elastic collisions with the singularity in each period. The proof relies on the Poincar\'e-Birkhoff Theorem. 
\end{abstract}

\maketitle

\section{Introduction} \label{sec:intro}

Differential equations with singularities appear as mathematical models in many scientific areas and have been studied from many viewpoints \cite{To}. In this paper, we consider the singular second order differential equation 
\begin{equation}\label{eq}
    \ddot u - \frac{1}{u^{\alpha}}=p(t), \ u>0,
\end{equation}
with parameter $\alpha>0$ and $p:\R\rightarrow\R$ a continuous and $2\pi$-periodic function. In a seminal paper, Lazer and Solimini~\cite{Lazer} proved that when $\alpha\geq 1$ equation~\eqref{eq} has a positive periodic solution if and only if $p$ has negative mean value. The authors also showed that the statement is sharp with respect to the parameter $\alpha$ in the sense that if $0<\alpha<1$, a function $p$ with negative mean value can be constructed in such a way that $\eqref{eq}$ has no periodic solutions. Later,  \cite[Example 3.9]{RTV} provided an effective sufficient condition over $p$ for the existence of a classical periodic solution in the weak repulsive case. The particular case $\alpha=1/2$ has been studied in~\cite{Rojas} showing that the equation corresponds to a perturbed isochronous oscillator and resonance conditions on the forcing term $p(t)$ are given.

In the mentioned references, existence of solutions is understood in the classical sense and collisions with the singularity are not allowed. The goal of the present paper is twofold. First, we aim to extend the notion of solutions of equation~\eqref{eq} for $0<\alpha<1$ admitting elastic collisions with the singularity at $x=0$. Second, we prove the existence of harmonic and sub-harmonic bouncing solutions of equation~\eqref{eq} for any negative $2\pi$-periodic forcing $p(t)$.

For the analogous equation with attractive nonlinearity (that is, changing the sign of the second term of the left-hand side of the equation), the notion of bouncing solution has been adequately defined and studied in a number of papers \cite{Or,QT,RS,RT,SQ,Tomecek,Zhao}. In contrast, it remains unexplored for the repulsive case. Our aim is to fill, at least partially, this gap. 

The structure of the paper is as follows. In Section 2, we analyze in detail the autonomous case (when the forcing term $p(t)$ is constant), including the associated period function and the continuation of colliding orbits. In Section 3, we define rigurously the notion of bouncing solution, proving that the initial boundary value problem is well-defined and continuable to the whole real line. Section 4 begins with the definition of the so-called successor map, which is a section of the flux whose fixed points are equivalent to periodic solutions of the equation. It can be proved that this map is area-preserving and a suitable version of the Poincar\'e-Birkhoff Theorem can be applied by using the estimates from Section 2, leading to the main results. 

\section{The integrable weak-singular system}

Throughout this section we consider a general potential function $V\in C^2(I)$ defined in an open interval $I=(\alpha,+\infty)$ satisfying
\[
\lim_{u\rightarrow \alpha^+}V(u)=h^*\ \text{ and }\ \lim_{u\rightarrow \alpha^+}V'(u)=-\infty.
\]
Additionally we assume that $u=0$ is the only local minimum of $V$. More precisely,
\[
V(0)=V'(0)=0,\ V''(0)>0\ \text{ and }\ uV'(u)>0 \text{ if }u\neq 0.
\]
Under these hypothesis it is clear that there exists $\beta>0$ with 
\[
\lim_{u\rightarrow \beta}V(u)=\lim_{u\rightarrow \alpha^+}V(u)=h^*>0,
\]
and the equation 
\begin{equation}\label{sode}
\ddot u+V'(u)=0, \ u\in I,
\end{equation}
has a center at the origin with a bounded period annulus, namely $\mathscr{P}$, which projection over the $u$-axis is the open interval $(\alpha,\beta)$ (see Figure~\ref{fig1}.) The associated first order differential system is a Hamiltonian system with Hamiltonian function $H(u,\dot u)=\frac{1}{2}\dot u^2+V(u)$. In particular, the energy at the outer boundary of the period annulus is $h^*$.
From the first integral $H$ we have that periodic orbits inside the period annulus correspond to energy levels $h\in(0,h^*)$. On the other hand, energies greater than $h^*$ correspond to solutions that collide with the singularity. More precisely, if $u(t)$ is a solution of~\eqref{sode} with $H(u,\dot u)=h\geq h^*$ and initial condition $\dot u(0)=0$, there exists $t_0>0$ such that $\lim_{t\rightarrow t_0^-}u(t)=\lim_{t\rightarrow -t_0^+}u(t)=\alpha.$ Moreover, by conservation of energy the limits $\lim_{t\rightarrow t_0^-}\dot u(t)=\dot u(t_0^-)$ and $\lim_{t\rightarrow -t_0^+}\dot u(t)=\dot u(-t_0^+)$ exist and the equality $\dot u(t_0^-)=-\dot u(-t_0^+)$ holds. That is, the collision with the singularity can be interpreted as an elastic collision. In particular, one can understand solutions with collisions as generalized periodic solutions, also known as bouncing periodic solutions. Indeed, the continuation is done by taking $\dot u(t_0^+)=-\dot u(t_0^-).$

\begin{figure}
    \centering
    \includegraphics[scale=1]{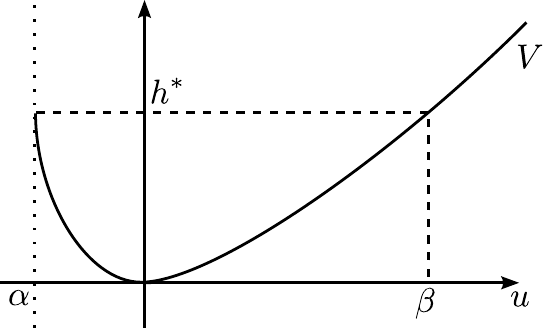} \hspace*{1cm}
    \includegraphics[scale=1]{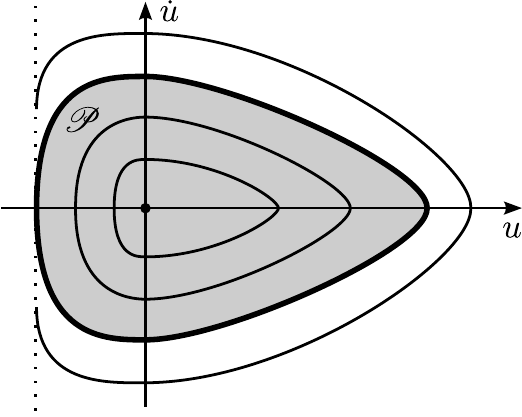}
    \caption{\label{fig1}On the left, potential function with a weak singularity at $x=\alpha$. On the right, the phase portrait of the potential system. The grey region corresponds to the period annulus. Its outer boundary is emphasized in bold.}
\end{figure}

\subsection{The period function and its extension to bouncing solutions}

The previous discussion shows that all solutions of equation~\eqref{sode} (classical and bouncing type) are periodic. The period function parametrized by the energy, $T_p:(0,h^*)\longrightarrow (0,+\infty)$, is the function that, for each $0<h<h^*$, assigns the period of the periodic solution $u(t)$ inside the energy level $h=H(u,\dot u)$. Due to the symmetry of the system with respect to the $u$-axis the function $T_p(h)$ is given by
\[
T_p(h)=\sqrt{2}\int_{u^-(h)}^{u^+(h)}\frac{du}{\sqrt{h-V(u)}},\ h\in(0,h^*),
\]
where $\alpha<u^{-}(h)<0<u^+(h)$ are the negative and positive solution of the equality $h-V(u)=0$. That is, the endpoints of the projection of the orbit over the $u$-axis.

When $h\geq h^*$ solutions of equation~\eqref{sode} are of bouncing type. In this case, we can define the time between two consecutive collisions of the solution $u(t)$ as
\[
T_b(h)=\sqrt{2}\int_{\alpha}^{u^+(h)}\frac{du}{\sqrt{h-V(u)}}, \ h\geq h^*.
\]
The function $T_b$ can be understood as an extension of the period function $T_p$ outside the period annulus. Therefore we define the function
\[
T(h)=\begin{cases}
T_p(h) & \text{if } 0<h<h^*,\\
T_b(h) & \text{if } h\geq h^*,
\end{cases}
\]
as the extended period function.

\begin{lem}\label{lema:period_C1}
The function $T(h)$ is $C^1(0,+\infty)$.
\end{lem}
\begin{proof}
The function $T_p$ is $C^1(0,h^*)$ by the classical theory of the period function. In addition, the function $T_b$ is $C^1(h^*,+\infty)$ since $V'(u)>0$ for all $u>0$ and $V(\alpha)<h^*$. Let us show that the result is true for $h=h^*$. To do so let us define 
\[
g(u)=\begin{cases}
u\left(\frac{V(u)}{u^2}\right)^{1/2} & \text{ if } u\neq 0,\\
0 & \text{ if } u=0.
\end{cases}
\]
The change of variable $u=g^{-1}(\sqrt{h}\sin\theta)$ transforms the integral expression of $T_b$ into
\[
T_b(h)=\sqrt{2}\int_{-\arcsin{\sqrt{\frac{h^*}{h}}}}^{\frac{\pi}{2}}(g^{-1})'(\sqrt{h}\sin\theta)d\theta.
\] 
The same change of variables transforms $T_p$ into
\[
T_p(h)=\sqrt{2}\int_{-\frac{\pi}{2}}^{\frac{\pi}{2}}(g^{-1})'(\sqrt{h}\sin\theta)d\theta.
\]
At this point the continuity becomes clear since $\lim_{h\rightarrow (h^*)^+}\arcsin{\sqrt{\frac{h^*}{h}}}=\frac{\pi}{2}.$ We differentiate with respect to $h$ to obtain
\[
T_b'(h)=\frac{1}{\sqrt{2h}}\int_{-\arcsin{\sqrt{\frac{h^*}{h}}}}^{\frac{\pi}{2}}(g^{-1})''(\sqrt{h}\sin\theta)\sin\theta d\theta - \frac{\sqrt{h^*}}{2h\sqrt{h-h^*}}\lim_{z\rightarrow{-\sqrt{h^*}}}(g^{-1})'(z).
\]
Using the identity $g^2=V$, we notice that
\[
(g^{-1})'(z)=\frac{1}{g'(g^{-1}(z))}=\frac{2g}{V'}(g^{-1}(z)).
\]
The previous equality together with $\lim_{u\rightarrow\alpha^+}g(u)=-\sqrt{h^*}$ implies that
\[
\lim_{z\rightarrow-\sqrt{h^*}}(g^{-1})'(z)=\lim_{u\rightarrow\alpha^+}\frac{2g(u)}{V'(u)}=0
\]
since $V'(u)\rightarrow -\infty$ as $u\rightarrow \alpha^+$. Consequently, on account of the previous limit together with the notation $u=g^{-1}(\sqrt{h}\sin\theta)$, the derivative of $T_b$ writes
\begin{equation}\label{dTb}
T_b'(h)=\frac{1}{\sqrt{2h}}\int_{-\arcsin{\sqrt{\frac{h^*}{h}}}}^{\frac{\pi}{2}}\frac{-g''(u)}{g'(u)^3}\sin\theta d\theta.
\end{equation}
This shows the continuity of the derivative at $h=h^*$.
\end{proof}

\begin{lem}\label{lema:periodoinfinito}
If $u/V'(u)\rightarrow+\infty$ as $u\rightarrow+\infty$ and $V''(u)>0$ for all $u>\beta$ then the function $T(h)$ tends to infinity as $h$ tends to infinity.
\end{lem}
\begin{proof}
For $h>h^*$ the extended period function writes
\[
T(h)=\sqrt{2}\int_{\alpha}^{\beta}\frac{du}{\sqrt{h-V(u)}}+\sqrt{2}\int_{\beta}^{u^+(h)}\frac{du}{\sqrt{h-V(u)}}\geq \int_{\beta}^{u^+(h)}\frac{du}{\sqrt{h-V(u)}} ,
\]
where $(\alpha,\beta)$ is the projection over the $u$-axis of the period annulus and $u^+(h)$ is the positive solution of $h-V(u)=0$ with $V'(u^+(h))>0$. By the mean value theorem, there exists $\beta<c(h)<u^+(h)$ such that
\[
\frac{1}{\sqrt{h-V(u)}} = \frac{1}{\sqrt{V'(c(h))}\sqrt{u^+(h)-u}},
\]
and so
\[
 \int_{\beta}^{u^+(h)}\frac{du}{\sqrt{h-V(u)}} = \frac{2\sqrt{u^+(h)-\beta}}{\sqrt{V'(c(h))}}.
\]
From the hypothesis $V''(u)>0$ for all $u>\beta$, we have $V'(c(h))<V'(u^+(h))$ and the previous equality yields to
\[
\int_{\beta}^{u^+(h)}\frac{du}{\sqrt{h-V(u)}} \geq \frac{2\sqrt{u^+(h)-\beta}}{\sqrt{V'(u^+(h))}}
\]
Since $\lim_{h\rightarrow+\infty}u^+(h)=+\infty$ and $\lim_{u\rightarrow+\infty}u/V'(u)=+\infty$ the previous inequality proves the result.
\end{proof}

The following result is an extension of Schaaf's monotonicity criterium~\cite{Schaaf} for the period function of planar potential systems with a weak singularity. The proof follows similarly to the original and here we only include some comments for the sake of brevity.

\begin{thm}\label{thm:schaaf}
If $V\in C^4(I)$ satisfies that
\begin{enumerate}[$(a)$]
    \item $(5(V''')^2-3V''V^{(4)})(u)>0$ for all $u\in I$ where $V''(u)>0$, 
    \item $V'(u)V'''(u)<0$ if $V''(u)=0$,
\end{enumerate}
then $T'(h)>0$ for all $h>0.$
\end{thm}

\begin{proof}
From Schaaf's monotonicity criterium~\cite{Schaaf} we already know that the hypothesis in the statement imply that $T_p(h)$ is monotone increasing on $(0,h^*)$. To show that $T_b(h)$ is monotone increasing on $(h^*,+\infty)$ we perform the change of variable $u=g^{-1}(\sqrt{h}\sin\theta)$ on the expression in~\eqref{dTb} and the derivative of $T_b$ writes
\begin{equation}\label{split}
\sqrt{2}hT_b'(h)=\int_{\alpha}^{\beta}\frac{\phi(u)V'(u)du}{\sqrt{h-V(u)}}+\int_{\beta}^{u^+(h)}\frac{\phi(u)V'(u)du}{\sqrt{h-V(u)}},
\end{equation}
where $\phi(u)=\frac{(V')^2-2VV''}{(V')^3}(u)$ and $(\alpha,\beta)$ is the projection of the period annulus over the $u$-axis. The original proof of Schaaf shows that the first integral is positive but the same arguments employed in~\cite[Lemma 1]{Schaaf} prove that $\phi(u)>0$ for $u>0$ and so the second integral is also positive.
\end{proof}

\begin{lem}\label{lema:monotona2}
If $(1-2VV''/(V')^2)(u)\geq \ell>0$ for all $u$ large enough, $u/V'(u)\rightarrow +\infty$ as $u\rightarrow+\infty$ and $V''(u)>0$ for all $u>\beta$ then there exists $\overline{h}>h^*$ such that $T(h)$ is monotone increasing on $(\overline{h},+\infty)$. 
\end{lem}

\begin{proof}
Denoting by $\overline{u}$ the point such that $(1-2VV''/(V')^2)(u)\geq \ell$ if $u\geq \overline{u}$, we split the expression~\eqref{split} in two parts:
\[
\sqrt{2}hT_b'(h)=\int_{\alpha}^{\overline{u}}\frac{\psi(u)du}{\sqrt{h-V(u)}}+\int_{\overline{u}}^{u^+(h)}\frac{\psi(u)du}{\sqrt{h-V(u)}},
\]
where $\psi(u)=(1-2VV''/(V')^2)(u)$. For the singular value $h=V(\overline{u})$ the first integral corresponds to the period of the bouncing solution with initial conditions $u(0)=\overline{u}$, $\dot u(0)=0$. That is, the integral is bounded for $h=V(\overline{u})$. Clearly this bound is uniform if $h$ is increased since the numerator is bounded on the fixed interval of integration and the denominator increases as the energy does. We claim that the second integral tends to infinity as $h$ tends to infinity. Indeed, we have
\[
\int_{\overline{u}}^{u^+(h)}\frac{\psi(u)du}{\sqrt{h-V(u)}}>\ell\int_{\overline{u}}^{u^+(h)}\frac{du}{\sqrt{h-V(u)}}.
\]
The integral at the right-hand side of the inequality tends to infinity as $h$ tends to infinity, as we have already shown in Lemma~\ref{lema:periodoinfinito}. Consequently, $\lim_{h\rightarrow +\infty}\sqrt{2}hT_b'(h)=+\infty$ and the result follows.
\end{proof}

\subsection{The power-like integrable system}\label{sec:integrable}

In this section we recover system~\eqref{eq} and analyze it when $p(t)\equiv p_0$ is a negative constant. That is,
\begin{equation}\label{eq_int}
\ddot u -\frac{1}{u^\alpha} = p_0,\ u>0.
\end{equation}
Equation~\eqref{eq_int} has an associated potential energy function given by
\[
V(u)\!:= - p_0 u - \frac{u^{1-\alpha}}{1-\alpha},\ u\geq 0.
\]
A direct study of the potential shows that the associated first order differential system of equation~\eqref{eq_int} exhibits a center at $((-p_0)^{-1/\alpha},0)$ with bounded period annulus, which projection over the $u$-axis is $(0,((\alpha-1)p_0)^{-1/\alpha})$. The total energy function is denoted by $H(u,\dot u)\!:=\frac{1}{2}\dot u^2+V(u).$ In particular, $h=0$ is the energy at the collision point with zero velocity (i.e. the energy at the outer boundary of the period annulus).

We denote by $u(t;u_0)$ the solution of~\eqref{eq_int} with initial conditions $u(0;u_0)=u_0>0$ and $\dot u(0;u_0)=0.$ Solutions with initial conditions $u_0\in(0,((\alpha-1)p_0)^{-1/\alpha})$ (that is, with negative energy) are globally defined and periodic. On the other hand, solutions with $u_0\geq ((\alpha-1)p_0)^{-1/\alpha}$ (positive energy) are no longer globally defined since they reach the singularity $u=0$ in finite time. From the expression of the energy function these orbits reach the singularity with finite velocity and the energy is conserved so the singularity can be interpreted as an elastic collision. More precisely, there exists $t_0>0$ such that $u(t_0;u_0)=u(-t_0;u_0)=0$ and $\dot u(t_0;u_0)=-\dot u(-t_0;u_0)$. 

\begin{lem}\label{lema:funcion_periodo}
Consider the extended period function $T(h)$ associated to equation~\eqref{eq_int} and any $p_0<0.$
\begin{enumerate}[$(a)$]
    \item If $\alpha>1/2$ then $T(h)$ is monotone increasing and tends to infinity.
    \item If $\alpha=1/2$ then $T(h)$ is constant for $h<0$, and monotone increasing and tends to infinity for $h>0$.
    \item If $0<\alpha<1/2$ then $T(h)$ is monotone decreasing for $h<0$ and there exists $\overline{h}>0$ such that $T(h)$ is monotone increasing in $(\overline{h},+\infty)$ and tends to infinity.
\end{enumerate}
\end{lem}

\begin{proof}
The result in $(a)$ follows applying Theorem~\ref{thm:schaaf}. Indeed, elementary computations lead to $V''(u)=\alpha u^{-(\alpha+1)}>0$ for all $u>0$ and 
\begin{equation}\label{Schaaf_exp}
(5(V''')^2-3V''V^{(4)})(u)=\alpha^2(\alpha+1)(2\alpha-1)u^{-2(\alpha+2)}.
\end{equation}
Therefore assumptions in Theorem~\ref{thm:schaaf} are fulfilled when $\alpha>1/2$ so the extended period function is monotone increasing. Moreover, 
\[
\lim_{u\rightarrow +\infty}\frac{u}{V'(u)}=\lim_{u\rightarrow+\infty}\frac{u}{-p_0-u^{-\alpha}}=+\infty
\]
since $\alpha>0$ and $p_0<0.$ Then Lemma~\ref{lema:periodoinfinito} implies that $T(h)$ tends to infinity.

To show $(c)$ we first employ the classical Schaaf's criterion in~\cite{Schaaf} for monotone decreasing period function. In this case the condition to be satisfied is expression in~\eqref{Schaaf_exp} to be negative. Clearly this is so for $0<\alpha<1/2$ and then $T(h)$ is monotone decreasing for $h<0$. For $h>0$ we notice that
\[
\lim_{u\rightarrow+\infty} \left(1-\frac{2VV''}{(V')^2}\right)(u)=\lim_{u\rightarrow+\infty} 1+\frac{2\alpha\left(p_0u+\frac{u^{1-\alpha}}{1-\alpha}\right)}{u^{1+\alpha}(p_0+u^{-\alpha})^{2}}=1,
\]
so assumptions in Lemma~\ref{lema:monotona2} hold and the result in $(c)$ follows. 

Finally, we prove $(b)$ by direct computation. Indeed, for $\alpha=1/2$ the integral of the expression of the period function writes
\[
\sqrt{2}\int\frac{du}{\sqrt{h-V(u)}}=\sqrt{2}\int\frac{du}{\sqrt{h+p_0u+2\sqrt{u}}}.
\]
The change of variables $u=z^2$ yields to
\[
\sqrt{2}\int\frac{du}{\sqrt{h-V(u)}}=2\sqrt{2}\int\frac{zdz}{\sqrt{h-\frac{1}{p_0}+p_0(z+1/p_0)^2}},
\]
which can be explicitly integrated, giving
\[
\frac{2\sqrt{2}}{(-p_0)^{\frac{3}{2}}}\arctan\left(\frac{\sqrt{-p_0}(z+\frac{1}{p_0})}{\sqrt{h-\frac{1}{p_0}+p_0(z+\frac{1}{p_0})}}\right)+\frac{2\sqrt{2}\sqrt{h-\frac{1}{p_0}+p_0(z+\frac{1}{p_0})}}{p_0}.
\]
It is then a computation to show that, if $h<0$, the evaluation of the previous function on both endpoints of the interval of integration give
\[
T_p(h)\equiv \frac{2\sqrt{2}\pi}{(-p_0)^{\frac{3}{2}}}.
\]
On the other hand, if $h>0$ the left-hand endpoint is $z=0$ and so
\[
T_b(h)=\frac{2\sqrt{2}}{(-p_0)^{\frac{3}{2}}}\left(\frac{\pi}{2}+\arctan((-p_0h)^{-\frac{1}{2}})\right)-\frac{2\sqrt{2h}}{p_0}.
\]
In particular, $\lim_{h\rightarrow 0^+}T_b(h)=2\sqrt{2}\pi(-p_0)^{-\frac{3}{2}}$. The properties on the statement are easily checked using the expression of $T_b$.
\end{proof}

\section{Regularization of collisions and bouncing solutions}

We now return to system~\eqref{eq}. The periodicity of $p(t)$ is not needed at this moment so we assume that $p:\R\rightarrow \R$ is a continuous and bounded negative function satisfying 
\begin{equation}\label{p_bounded}
    p_2\leq p(t)\leq p_1<0 
\end{equation}
for all $t\in\R.$

Let us consider the first order differential system associated to~\eqref{eq}. That is,
\[
X: \dot u = v, \ \dot v = \frac{1}{u^{\alpha}}+p(t).
\]
We also denote by $X_1$ and $X_2$ the first order differential systems associated to \eqref{eq} taking $p(t)\equiv p_1$ and $p(t)\equiv p_2$, respectively. Notice that $X_1$ and $X_2$ are both integrable first order differential systems associated to an equation of the form~\eqref{eq_int}. We denote by $H_1$ and $H_2$ the energy functions associated to $X_1$ and $X_2$. We also define $\eta=((\alpha-1)p_1)^{-1/\alpha}$ for convenience. 

\begin{lem}\label{lema:choque}
Let us assume that $p(t)$ is a continuous and bounded negative function satisfying~\eqref{p_bounded}. Then every classical solution $u$ of equation~\eqref{eq} with initial conditions $u(0)>\eta$ and $\dot u(0)=0$ has a finite maximal interval of definition $(t_0,t_1)$ such that $u(t_0)=u(t_1)=0$. Moreover, $\dot u(t_0)>0$ and $\dot u(t_1)<0$ are finite.
\end{lem}

\begin{proof}
Let $u(t)$ be a maximal solution of equation~\eqref{eq} with initial conditions $u(0)=u_0>0$ and $\dot u(0)=0$ and let $u_1(t)$ and $u_2(t)$ be the maximal solutions of equation~\eqref{eq} with same initial conditions as $u(t)$ taking $p(t)\equiv p_1$ and $p(t)\equiv p_2$, respectively. Since $p_1$ and $p_2$ are negative and the initial conditions satisfy $u_i(0)>\eta$ and $\dot u_i(0)=0$, $i=1,2$, by the discussion of the previous section, both $u_1$ and $u_2$ are solutions defined in a bounded interval, $(t_{01},t_{11})$ and $(t_{02},t_{12})$ respectively, and they reach the singularity. 
The function $u_1$ is contained in the level curve $H_1(u,v)=\frac{1}{2}v^2-p_1u-\frac{u^{1-\alpha}}{1-\alpha}=H_1(u_0,0)$ whereas $u_2$ is contained in the level curve $H_2(u,v)=\frac{1}{2}v^2-p_2u-\frac{u^{1-\alpha}}{1-\alpha}=H_2(u_0,0)$. Notice that $H_2(u_0,0)-H_1(u_0,0)=(p_1-p_2)u_0>0.$

First we show that $t_{12}<t_{11}$. To do so let us argue by contradiction assuming that $t_{11}\leq t_{12}$. Since $u_1(0)=u_2(0)>\eta$ and $\dot u_1(0)=\dot u_2(0)=0$ the difference function $\omega(t)=u_1(t)-u_2(t)$ satisfies $\omega(0)=0$, $\dot \omega(0)=0$ and $\ddot \omega(t)=\frac{1}{u_1(t)^{\alpha}}-\frac{1}{u_2(t)^{\alpha}}+p_1-p_2.$ In particular, $\ddot \omega(0)=p_1-p_2>0$ and so $u_1(t)>u_2(t)$ for small $t$ positive. Therefore, if $t_{11}\leq t_{12}$, there exists some $0<t^*\leq t_{11}$ such that $u_1(t^*)=u_2(t^*)$ with $|\dot u_1(t^*)|>|\dot u_2(t^*)|$. Using the energy level of each solution, we have that $H_2(u_2(t^*),\dot u_2(t^*))-H_1(u_1(t^*),\dot u_1(t^*))=(p_1-p_2)u_0.$ That is,
\[
\frac{1}{2}(\dot u_2(t^*)^2-\dot u_1(t^*)^2)=(p_1-p_2)(u_0-u_1(t^*))>0.
\]
Then $\dot u_2(t^*)^2>\dot u_1(t^*)^2$ reaching contradiction. Similarly one can show that $t_{01}<t_{02}$. Therefore $(t_{02},t_{12})\subset (t_{01},t_{11})$.

Let us show now that the maximal interval of definition of $u(t)$ is $(t_0,t_1)$ finite satisfying 
\begin{equation}\label{tiempos}
t_{01}\leq t_0\leq t_{02}<0<t_{12}\leq t_1\leq t_{11}
\end{equation}
and that $u_2(t)\leq u(t)\leq u_1(t)$ for all $t$ in the common interval of definition. Indeed, $\left<X,\nabla H_1\right>=v(p(t)-p_1)$ which, since $p(t)\leq p_1$, has opposite sign than $v$. Similarly, $\left<X,\nabla H_2\right>=v(p(t)-p_2)$ has the same sign than $v$. In particular, the trajectory $(u(t),\dot u(t))$ is confined in the region delimited by $H_1(u,v)=H_1(u_0,0)$ and $H_2(u,v)=H_2(u_0,0)$ for all time $t$ in the interval of definition of $u(t)$. We point out that the outer boundary of the region is given by $H_2(u_0,0)$, whereas the inner boundary is given by $H_1(u_0,0)$. Moreover, at $(u_0,0)$ the vector field $X$ is vertical and points down. This implies that the function $u(t)$ for $t\geq 0$ is decreasing and $\dot u(t)$ cannot tend to zero. Thus $u(t)$ reaches the singularity $u=0$ in finite time and $-\infty<\dot u_1(t_{11})\leq \dot u(t_1)\leq \dot u_2(t_{12})<0$. Finally assume that $t_1>t_{11}$. Then there exists $t^*$ in such a way $u(t^*)=u_1(t^*)$ and $\dot u(t^*)>\dot u_1(t^*)$. This contradicts the fact that $u(t)$ is inside the region mentioned before. Thus $t_1\leq t_{11}$. Respective arguments with $u_2$ shows $t_1\geq t_{12}$. The result holds backwards in time similarly.
\end{proof}

\begin{lem}\label{lema:regularizacion}
Let $t_0$ and $v_0$ be two numbers with $v_0>\sqrt{2(p_1-p_2)\eta}$. Assume that $p(t)$ is Lipschitz-continuous. Then there exists a unique maximal solution of~\eqref{eq} defined in $(t_0,t_1)$ with $t_0<t_1<+\infty$ satisfying 
\[
\lim_{t\rightarrow t_0^+}u(t)=\lim_{t\rightarrow t_1^-}u(t)=0
\]
and
\[
\lim_{t\rightarrow t_0^+}\dot u(t)=v_0 \ \text{ and }\ \lim_{t\rightarrow t_1^-}\dot u(t)= v_1,
\]
for some real number $v_1<0.$
\end{lem}

\begin{proof}
Let us consider the system
\[
\dot u = +\sqrt{2(w-V(u))},\ 
    \dot w = p(t)\sqrt{2(w-V(u))}.
\]
For the classical theory of differential equations the Cauchy problem $u(0)=0$, $w(0)=h_0$ has a local unique solution if $h_0>0=V(0)$. That function $u=u(t)$ is defined in some open interval $(0,t^*)$ in which it is also solution of the equation~\eqref{eq} with initial conditions $u(0)=0$ and $\dot u(0)=+\sqrt{2h_0}$. Indeed, from the first equation of the previous system,
\[
\frac{1}{2}\dot u^2=w-V(u).
\]
Multiplying by $\dot u$ and integrating the equation we get
\[
\ddot u(t) \dot u(t) = (p(t)-V(u))\dot u(t),
\]
so we recover equation~\eqref{eq} due to $\dot u(t)\neq 0$ for $t\in(0,t^*)$. The previous system then acts as a regularization of the collision of equation~\eqref{eq}. Indeed, by uniqueness of the initial value problem equation~\eqref{eq} has a unique solution defined in $(0,t^*)$ and coinciding with $u(t)$ in that interval. This is enough to ensure that they coincide everywhere in the interval of definition $(t_0,t_1)$. Here $t_1$ may be infinite. Arguing similarly as in the proof of Lemma~\ref{lema:choque}, the condition on $v_0$ in the statement implies that the solution reach a local maximum $u^*>\eta$. Therefore, the solution of the equation~\eqref{eq} satisfies the assumptions in Lemma~\ref{lema:choque} and so $t_1$ is finite and $u(t)$ has a collision with finite velocity at $t=t_1$.
\end{proof}

The previous results allow to define a~\emph{bouncing solution} of~\eqref{eq} as a continuous function $u:\R\rightarrow[0,+\infty)$ satisfying
\begin{enumerate}[$(a)$]
    \item $Z=\{t\in\R : u(t)=0\}$ is discrete,
    \item for any open interval $I\subset\R\setminus Z$ the function $u$ is in $C^2(I)$ and satisfies equation~\eqref{eq} on $I$,
    \item for each $t_0\in Z$ the limits $\lim_{t\rightarrow t_0^+}\dot u(t)=\dot u(t_0^+)$ and $\lim_{t\rightarrow t_0^-}\dot u(t)=\dot u(t_0^-)$ exist and satisfy $\dot u(t_0^+)=-\dot u(t_0^-)$.
\end{enumerate}

\begin{rem}
We point out that the last item implies that the limit 
\[
\lim_{t\rightarrow t_0}H(u(t),\dot u(t))=h_0
\]
exists. This limit is taken from both sides of $t_0$ and hence the energy function has a well defined value at $t=t_0$. That is, the energy is preserved at the collision. In all arguments velocity and energy play an analogous role. For instance, condition $v_0>\sqrt{2(p_1-p_2)\eta}$ can be replaced by $h_0>(p_1-p_2)\eta$.
\end{rem}

\begin{figure}
    \centering
    \includegraphics[scale=1]{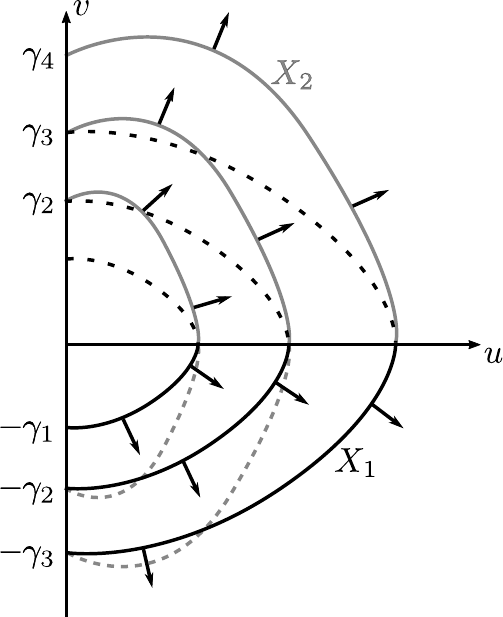}
    \caption{\label{fig2}Scheme of the proof of Lemma~\ref{lema:n-impacts} for $n=4$. Lines in black (bold and dashed) correspond to level curves of the differential system $X_1$. Lines in grey (bold and dashed) correspond to level curves of the differential system $X_2$. Arrows represent the direction of the vector field $X$ on the level curves of $X_1$ and $X_2$.}
\end{figure}

Condition $v_0>\sqrt{2(p_1-p_2)\eta}$ is enough to ensure the occurrence of at least one collision, but subsequent impacts are not ensured. The following lemma implies that solutions with sufficiently high energy exhibit an arbitrary number of collisions.

\begin{lem}\label{lema:n-impacts}
 Assume that $p(t)$ is Lipschitz-continuous. For any $n\in \mathbb{N}$, there exists $\gamma_n>0$ such that if $t_0$ and $v_0$ are two numbers with $v_0>\gamma_n$, the unique solution of~\eqref{eq} has at least $n$ impacts.
\end{lem}
\begin{proof}
Let us take $\gamma_1>\sqrt{2(p_1-p_2)\eta}$ and consider the solution of $X_1$ that reaches the point $(0,-\gamma_1)$. This solution crosses the $v-$axis backwards in time at some $u=u_1>0$. Now consider the solution of $X_2$ that crosses the point $(u_1,0)$. This solution collides with the singularity $u=0$ backwards in time with some velocity $v=\gamma_2$.

First, we point out that $\gamma_2>\gamma_1$. Indeed,  $H_2(u_1,0)-H_1(u_1,0)=(p_1-p_2)u_1>0$ and $H_1(u_1,0)=H_1(0,\gamma_1)=\frac{1}{2}\gamma_1^2$ and $H_2(u_1,0)=H_2(0,\gamma_2)=\frac{1}{2}\gamma_2^2$. In particular, $\gamma_2>\sqrt{2(p_1-p_2)\eta}$ and Lemma~\ref{lema:regularizacion} implies that the solution of~\eqref{eq} with $u(0)=0$ and $\dot u(0)=v_0>\gamma_2$ has at least one collision.

Second, since $\left<X,\nabla H_1\right>=v(p(t)-p_1)$ and $\left<X,\nabla H_2\right>=v(p(t)-p_2)$, property~\eqref{p_bounded} implies that the solution of~\eqref{eq} impacts the singularity with velocity $-v_1<-\gamma_1$. Thus the bouncing solution is continued by the solution of~\eqref{eq} with initial conditions $u(0)=0$, $\dot u(0)=v_1>\gamma_1$, which ensures at least one collision more. Thus the solution of~\eqref{eq} with $u(0)=0$ and $\dot u(0)=v_0>\gamma_2$ has at least two collisions.

This procedure generates a succession $\gamma_1<\gamma_2<\dots<\gamma_n$ and the solution of~\eqref{eq} with initial conditions $u(0)=0$ and $\dot u(0)=v_0>\gamma_n$ has at least $n$ collisions with the singularity.
\end{proof}

\section{The successor map and the twist condition}

Now we recover the periodicity property of the forcing term $p(t)$. For convenience, let us call $\gamma\!:=\sqrt{2(p_1-p_2)\eta}$. For a given $t_0\in \R$ and $v_0>\gamma$, Lemma \ref{lema:regularizacion} assures that there exists a unique solution $u(t)$ of ~\eqref{eq} such that $u(t_0)=0,\dot u(t_0)=v_0$. Moreover, such solution has a finite interval of definition and vanishes at some time $t_1$. We define the successor map
$$
S:\R\times (\gamma,+\infty)\to\R\times\R^+
$$
$$
(t_0,v_0)\to S(t_0,v_0)=(t_1,-\dot u(t_1^-)).
$$
In the following, we denote $S_1(t_0,v_0)=t_1,\,S_2(t_0,v_0)=-\dot u(t_1^-)$. The map $S$ is one-to-one and continuous in its domain. Moreover, by the $2\pi$-periodic dependence of the equation, one has
$$
S(t_0+2\pi,v_0)=S(t_0,v_0)+(2\pi,0),
$$
and then $(t_0,v_0)$ can be seen as polar coordinates. 

\subsection{The generalized Poincar\'e-Birkhoff Theorem}\label{sub4.1} For completeness, in this subsection we enunciate the version of the Poincar\'e-Birkhoff Theorem that is used in the proofs. This version was presented in \cite{QT} as a variant of the main result of \cite{F}.

\begin{thm}[twist theorem] 
Let $A={\bf S}^1\times[a_1,a_2]$, $B={\bf S}^1\times[b_1,b_2]$ be two annuli in the plane with $A\subset B$. Assume that $f:A\to B$ is an area-preserving homeomorphism such that the area of the two connected components of the complement of $f(A)$ in B is the same as the area of the corresponding components of the complement of A in B. Assume also that $f$ has a lift $\tilde f:\R\times[a_1,a_2]\to \R\times [b_1,b_2]$ of the form
$$
\theta'=\theta+h(\theta,\rho), \quad \rho'=g(\theta,\rho),
$$
where $h,g$ are continuous and $2\pi$-periodic in $\theta$. Then, if the boundary twist condition 
$$
h(\theta,a_1)\cdot h(\theta,a_2)<0 \mbox{ for any } \theta\in [0,2\pi]
$$
holds, $f$ has at least two geometrically distinct fixed points.
\end{thm}

\subsection{The successor map is area-preserving}

Consider a sequence $\{\epsilon_n\}_{n\geq 0}$ with $\epsilon_n>0$ and $\epsilon_n\downarrow 0$, and for each $n\geq 0$ the second order differential equation
\begin{equation}\label{regular_eq}
    \ddot u - \frac{1}{(u+\epsilon_n)^{\alpha}}=p(t),
\end{equation}
which is a translation on the $u$-axis of the original equation~\eqref{eq} with the singularity placed at $u=-\epsilon_n$. Every solution $u_{n}$ of equation~\eqref{regular_eq} is a translation of a solution of equation~\eqref{eq}, $u_n(t)=u(t)-\epsilon_n$. By Lemma~\ref{lema:choque} every classical solution $u_n$ with initial conditions $u_n(0)=u(0)-\epsilon_n$ with $u(0)>\eta$ and $\dot u_n(0)=0$ has a finite maximal interval of definition and reach the singularity $u=-\epsilon_n$ forwards and backwards in time with finite velocity. In particular, there exist times $t_{0n}$ and $t_{1n}$ in which $u_n(t_{0n})=u_n(t_{1n})=0$ satisfying $t_0<t_{0n}<0<t_{1n}<t_1$, where $(t_0,t_1)$ is the maximal interval of definition of the classical solution of equation~\eqref{eq} with initial conditions $u(0)>\eta$, $\dot u(0)=0$. Continuous dependence on $\epsilon_n$ of $u_n(t)$ is deduced from the explicit formula $u_n(t)=u(t)-\epsilon_n$. Let us define $w_n(t)$ as $u_n(t)$ for $t\in(t_{0n},t_{1n})$ and identically zero for $t\in(t_0,t_{0n}]\cup[t_{1n},t_1)$. With all the previous comments in mind the following lemma is straightforward.

\begin{lem}\label{lema:convergencia}
The sequences $\{w_n\}_{n\geq 0}$ and $\{\dot w_n\}_{n\geq 0}$ tend respectively to $u$ and $\dot u$ uniformly on $(t_0,t_1)$. Moreover, $t_{0n}\rightarrow t_0$ and $t_{1n}\rightarrow t_1$ as $n\rightarrow +\infty$.
\end{lem}

Let us now consider the sequence of successor mappings corresponding to functions $w_n$. That is, for each $(t_0,v_0)\in\R\times(\gamma,+\infty)$ we define
\[
S_n(t_0,v_0)=(t_{1n},-\dot w_n(t_{1n})).
\]
It is clear from Lemma~\ref{lema:convergencia} that $\{S_n\}_{n\geq 0}$ converges point-wise to $S$. A time reversion argument also shows that $\{S_n^{-1}\}_{n\geq 0}$ converges point-wise to $S^{-1}$. The extension of the area-preserving property of $\{S_n\}_{n\geq 0}$ to $S$ is now verbatim the case in~\cite{Tomecek}, what leads to the following statement.

\begin{prop}
The successor mapping $S$ is area-preserving.
\end{prop}

\subsection{Existence of bouncing periodic orbits}

Now, given natural numbers $m,n$, our objective is to find fixed points of the map
$$
S^n(t_0,v_0)-(2m\pi,0),
$$
that are identified as the initial conditions of $2m\pi$-periodic solution with exactly $n$ impacts in each period.

Our first main result is the following one.

\begin{thm}\label{th1}
Assume that $0<\alpha<1$ and $p(t)$ is a Lipschitz-continuous and $2\pi$-periodic function with negative values. Then, there exists $m_1\in\N$ such that for any $m\geq m_1$, equation~\eqref{eq} has at least two $2m\pi$-periodic solutions with exactly $1$ impact in the period interval $[0,2m\pi)$.
\end{thm}

The proof relies the version of Poincar\'e-Birkhoff Theorem presented in subsection \ref{sub4.1} and the following key technical lemma.

\begin{lem}\label{lema-twist}
\[
\lim_{v_0\to+\infty}S_1(t_0,v_0)-t_0=+\infty\qquad\mbox{ uniformly in }t_0\in [0,2\pi].
\]
\end{lem}

\begin{proof}

From the equation and \eqref{p_bounded},
\[
\ddot u(t)>p_2
\]
for all $t\in (t_0,t_1)$. Then, an integration from $t_0$ to $t_1$ gives
\[
\dot u(t_1^-)-v_0>p_2(t_1-t_0).
\]
Now, considering that $\dot u(t_1^-)<0$, we have
\[
v_0<-p_2(t_1-t_0),
\]
and the conclusion is clear if we remember that $p_2<0$ {\text{blue} and $S_1(t_0,v_0)=t_1$}.
\end{proof}

{\noindent\it Proof of Theorem \ref{th1}.} By continuity, we can fix $m_1$ such that 
\begin{equation}\label{desingualdad}
S_1(t_0,\gamma+1) -t_0<2 m_1\pi
\end{equation}
for any $\tau\in [0,2\pi]$. 
Now, for a given $m\geq m_1$, by Lemma \ref{lema-twist} there exists $v_+>\gamma+1$ such that
\[
S_1(t_0,v_+) -t_0>2 m\pi
\]
for any $t_0\in [0,2\pi]$. Now, the result is a direct consequence of the Poincar\'e-Birkhoff Theorem. \qed 

One of the main differences with the attractive case studied in \cite{Or,QT} is that the successive iterations of $S$ are not necessarily well-defined in the repulsive case. However, Lemma~\ref{lema:n-impacts} proves the existence of $\gamma_n$ such that the $n$-th iterate $S^n(t_0,v_0)$ is well-defined for all $v_0>\gamma_n$. With this observation in mind, we can prove the following result.

\begin{thm}\label{th2}
Assume that $0<\alpha<1$ and $p(t)$ is a continuous and $2\pi$-periodic function with negative values. Then, for any natural number $n\geq 2$, there exists $m_n\in\N$ such that, for any $m\geq m_n$, equation \eqref{eq} has at least one $2m\pi$-periodic solution with exactly $n$ impacts in the period internal $[0,2m\pi)$.
\end{thm}

\begin{proof}
By the observation above, $S^n(t_0,v_0)$ is well-defined for $v_0>\gamma_n$. Then we can follow exactly the same argument as in the proof of Theorem \ref{th1}, working now with the first component of $S^n(t_0,v_0)$. Again, Poincar\'e-Birkhoff Theorem provides two fixed points of the successor map, the difference is that they may correspond to the same bouncing solution.
\end{proof}

The previous results ensure the existence of sub-harmonic bouncing solutions. In order to guarantee the existence of harmonic bouncing solution we need an accurate statement.

\begin{thm}\label{th3}
Assume that $0<\alpha\leq 1/2$ and $p(t)$ is a continuous and $2\pi$-periodic function satisfying $p(t)\leq p_1<0$. If $\left(\frac{1}{\alpha}\right)^{\frac{\alpha}{1+\alpha}}<-p_1$ then equation~\eqref{eq} has at least two $2\pi$-periodic solutions with exactly $1$ impact in the period interval $[0,2\pi)$.
\end{thm}
\begin{proof}
The proof is verbatim the one of Theorem~\ref{th1} with the difference that in this case we need to ensure that the inequality~\eqref{desingualdad} is verified for $m_1=1.$ To do so we first notice that, due to Lemma~\ref{lema:choque}, for any 
$(t_0,v_0)\in\R\times (\gamma,+\infty)$ the quantity $S_1(t_0,v_0)-t_0=t_1-t_0$ is bounded by the length of the interval of definition of the solution of~\eqref{eq} with $p(t)\equiv p_1$ (see~\eqref{tiempos}). Therefore it is enough to prove that the integrable equation with $p(t)\equiv p_1$ has a bouncing solution with period less than $2\pi$.

Equation~\eqref{eq} with $p(t)\equiv p_1$ has a center located at $((-p_1)^{-\frac{1}{\alpha}},0)$ (see Section~\ref{sec:integrable}) and the value of the period function at the center itself is given by
\[
\frac{2\pi}{\sqrt{V''\bigr((-p_1)^{-\frac{1}{\alpha}}\bigl)}}=\frac{2\pi}{\sqrt{\alpha(-p_1)^{\frac{1+\alpha}{\alpha}}}}.
\]
For $0<\alpha\leq 1/2$, Lemma~\ref{lema:funcion_periodo} guarantees that the period function is either constant or monotone decreasing inside the period annulus. Therefore, bouncing solutions of equation~\eqref{eq} with $p(t)\equiv p_1$ near the boundary of the period annulus have bouncing period less or equal than $2\pi/\sqrt{\alpha(-p_1)^{\frac{1+\alpha}{\alpha}}}$ by a continuity argument (see Lemma~\ref{lema:period_C1}). The condition in the statement implies that the previous value is less than $2\pi$.
\end{proof}

The fact that the period function for $1/2<\alpha<1$ is monotone increasing do not allow to ensure the existence of harmonic bouncing solutions using the value of the period at the center. In order to derive an analogous result for this range of $\alpha$ we need the expression of the period at the outer boundary of the period annulus. However the authors are not able to obtain its expression with their methods available.

\section*{Acknowledgements}
We are grateful to Rafael Ortega for bringing to our attention the position-energy system that has been crucial for the regularization of collisions. This work has been realized thanks to the \emph{Agencia Estatal de Investigación} and \emph{Ministerio de Ciencia, Innovación y Universidades} grants MTM2017-82348-C2-1-P and MTM2017-86795-C3-1-P. 
 
\bibliographystyle{abbrv}

\end{document}